\documentclass[11pt]{amsart}
\usepackage{color ,epstopdf ,amssymb ,graphicx ,geometry,verbatim}     
\usepackage{longtable,verbatim}
           
\newtheorem{proposition}{Proposition}
\newtheorem{theorem}{Theorem}
\newtheorem{problem}{Problem}
\newtheorem{corollary}{Corollary}

\title{Core-free, rank two coset geometries from edge-transitive bipartite graphs}
\author{Julie De Saedeleer} \address{Universit\'e Libre de Bruxelles
D\'epartement de Math\'ematiques - C.P.216,
Boulevard du Triomphe,
B-1050 Bruxelles} \email[J.~De Saedeleer]{judesaed@ulb.ac.be}
\author{Dimitri Leemans} \address{Universit\'e Libre de Bruxelles
D\'epartement de Math\'ematiques - C.P.216,
Boulevard du Triomphe,
B-1050 Bruxelles}
\author{Mark Mixer} \address{Universit\'e Libre de Bruxelles
D\'epartement de Math\'ematiques - C.P.216,
Boulevard du Triomphe,
B-1050 Bruxelles}
\author{Toma\v z Pisanski} \address{Slovenija University of Ljubljana, 
In\v stitut za matematiko,  fiziko in mehaniko
Jadranska 19,
Ljubljana}

\date{\today}                                          
\begin{document}
\maketitle

\begin{abstract}
It is known that the Levi graph of any rank two coset geometry is an edge-transitive graph, and thus coset geometries can be used to construct many edge transitive graphs.  In this paper, we consider the reverse direction.  Starting from edge-transitive graphs, we construct all associated core-free, rank two coset geometries.  
In particular, we focus on 3-valent and 4-valent graphs, and are able to construct coset geometries arising from these graphs.
We summarize many properties of these coset geometries in a sequence of tables; in the 4-valent case we restrict to graphs that have relatively small vertex-stabilizers.

\end{abstract}



\section{Introduction}

Most of the time, people working with coset geometries are either analyzing a particular example of a geometry, or taking a group (or a family of groups) with the idea of classifying all geometries, satisfying a set of axioms, where the group has a flag-transitive action \cite{BCDL2001,DSL10}.  In this paper we use these ideas in the reverse direction. 
Namely, instead of considering flag-transitive incidence geometries arising from a given group, we use lists of edge-transitive graphs, and study flag-transitive incidence geometries arising from subgroups of their automorphism groups. In particular, we are interested in finding two-transitive geometries.  Our computations were
 made on well-known families of cubic and quartic bipartite graphs. Our approach uses the well-established pathway that was initiated in the 1950's and 1960's 
 by Jacques Tits~\cite{Tit54,Tit62a}.

This paper addresses multiple fields of mathematics including: group theory, geometry, algebraic graph theory, and combinatorics.  For completeness, we recall several known results in each of these fields.  The paper is organized as follows. 
First we give  definitions and notation for groups, graphs, and geometries.
Subsequently, we provide  background results for groups, graphs, and geometries.
We also give main observations about flag-transitive incidence geometries and coset geometries.  Finally, we provide explanations of our computations, and give some observations regarding the information in our tables, with a link to a web page of data and results.

\subsection{Groups.}
We recall some basic notions about groups and give the definitions needed to understand this paper.

The action of a group $G$ on a set $X$ is called \textit{transitive} if there is only one orbit; it is called \textit{primitive} if it is transitive and there is no non-trivial partition of $X$ preserved by the action.
The action is \textit{doubly transitive} if any ordered pair of distinct elements can be mapped pointwise on any other ordered pair by an element of $G$. 

An action is called \textit{faithful} if for any two distinct $g$, $h$ in $G$ there exists an $x$ in $X$ such that $gx \neq hx$; or equivalently, if for any $g \neq e$ in $G$ there exists an $x$ in $X$ such that $gx \neq x$. Intuitively, different elements of $G$ induce different permutations of $X$.

We denote the stabilizer of an element $p$ of $X$ by $Stab(G,p)$, which is the subgroup of $G$ that fixes $p$.  A \textit{maximal subgroup} of a group is defined as a proper subgroup such that there 
is no other proper subgroup containing it. 

\subsection{Bipartite graphs and their symmetry.}

All graphs in this paper are simple, i.e. without loops or multiple edges. Let $X$ be a graph and $G$ a 
subgroup of its automorphism group. We consider each edge of
$X$ being composed of two opposite \textit{arcs} (sometimes called \textit{darts}, \textit{half-edges}, or \textit{semi-edges}).
A \textit{$k$-arc} is an ordered $(k+1)$-tuple $(\alpha_0, ... ,\alpha_k)$ of vertices such that 
      $\{\alpha_{i-1},\alpha_{i} \}$ is an edge of $X$ for all $i=1,...,k$ and $\alpha_{j-1} 
      \neq \alpha_{j+1}$ for all $j=1,...,k-1$. (see~\cite{Big93}); note that a 1-arc is the same as an arc.

We now give the following definitions.
\begin{enumerate}
\item If the group $G$ acts transitively on the set of edges then we call $X$ a \textit{$G$-edge transitive graph}.
\item If the group $G$ acts transitively on the set of arcs then we call $X$ a \textit{$G$-arc transitive graph}.
\item $X$ is \textit{$G$-$k$-arc transitive} if $G$ acts transitively on $k$-arcs, but not on the $k+1$-arcs (see~\cite{Big93}).
\item $X$ is \textit{$G$-half-arc transitive} if $G$ acts transitively on the vertices and on the edges, but not on the arcs (see~\cite{Bou70,Mar98}).
\item $X$ is \textit{$G$-semi-symmetric} if $G$ acts transitively on arcs but not transitively on vertices (see~\cite{MPSW2007}).
\end{enumerate}

Note that our definition of a semi-symmetric graph is more general than usual. Namely, we do not require that a semi-symmetric graph be regular. Also note that saying a graph $X$ is $G$-arc-transitive is the same as saying that $X$ is $G$-$k$-arc-transitive for some $k \geq 1$. 
The notation of $G$-$k$-arc-transitive, is the original notation, it also appears noted as $(G,k)$-arc-transitive.

Let $X$ be a graph with a given vertex coloring.  We define three groups of automorphisms of $X$.

  The \textit{full automorphism group}, $Aut(X)$ consists of all automorphisms of $X$.
The  group of \textit{color respecting automorphisms}, $Aut_O(X)$, consists of automorphisms of $X$ which preserve the partition of the vertices of $X$ into color classes. The  group of \textit{color preserving automorphisms}, $Aut_o(X)$, consists of automorphisms of $X$ such that a vertex and its image are in the same color class.

An automorphism of a bipartite graph $X$ with a given black-and-white coloring that maps black vertices to white vertices is called a \emph{duality}.  A duality of order 2 is called \textit{polarity}. If a bipartite graph admits a duality, then it is called \emph{self-dual}; similarly if it admits a polarity, then it is called \emph{self-polar}.  

Observe that similar concepts can be defined for edge colored graphs.  

However, in this paper we only consider bipartite graphs with a given (black-and-white) vertex coloring.

\subsection{Rank two incidence and coset geometries.}
The theory of incidence geometries for general rank has existed for about half a century (see~\cite{Bue86,Bue95a,Bue95,BCD96a,BP95,Lee97c, Tit80}). 
However, we exclusively consider rank two geometries and adjust the language accordingly.

A rank two \emph{incidence geometry} $\Psi = (P,L,I)$ is a structure consisting of a set $P$ whose elements are called  \emph{points}, a set $L$ whose elements are called 
\emph{lines} and an \emph{incidence relation}  $I$ between points and lines. We model such a geometry using a bipartite graph with a given black and white coloring.  Black vertices correspond to points and white vertices correspond to lines of the geometry. This graph is called the \emph{incidence graph} or \emph{Levi graph} of the geometry (see~\cite{Lev42}, pg 5). 
By interchanging the roles of points and lines in $\Psi$ we obtain the \textit{dual geometry} $\Psi^* = (L,P,I)$. This is equivalent to reversing the colors in the 2-coloring of the Levi graph.  The edges of this graph are called \emph{chambers}~\cite{Tit62a} or \emph{flags}~\cite{MS2002} of the geometry.   A morphism between two geometries maps points to points, lines to lines, and flags to flags.  An isomorphism is a bijective morphism. An isomorphism
of $\Psi$ to itself is called an automorphism, and the set of automorphisms form a group, which we denote $Aut_o(\Psi)$. We point out here that this is not the standard notation for incidence geometries, but it allows us to easily go back and forth between automorphisms of graphs and rank two geometries.  An isomorphism of $\Psi$ to its dual geometry $\Psi^*$ is called a duality.
 Finally, an incidence geometry is \textit{connected} if its Levi graph is connected. 

A remarkable concept, whose power was discovered by Tits in the 1950's, comes from the requirement that $Aut_o(\Psi)$ acts transitively on the flags of the geometry.  We call a geometry $\Psi$ \textit{point-, line-, or flag-transitive} provided that $Aut_o(\Psi)$ acts transitively respectively on points, lines, or flags of the geometry (see~\cite{Deh94}).  It may happen that the group of automorphisms and dualities acts transitively on flags but the geometry itself is not flag-transitive. In this case we say, that it is \textit{weakly flag-transitive}~\cite{MarusicPisanski}.

Flag-transitive incidence geometries arise from groups and their coset geometries.  A rank two \textit{coset geometry} $\Gamma=(G:G_0,G_1)$ is defined by a group $G$ and left cosets of two of its subgroups $G_0$ and $G_1$.  Given a coset geometry $\Gamma$, we define its \textit{underlying incidence geometry} $\Psi$ as follows: points of $\Psi$ are left-cosets of $G_0$ in $G$ and lines are left-cosets of $G_1$ in $G$.
A pair of cosets $aG_0,bG_1$ is incident if and only if $aG_0 \cap bG_1 \neq \emptyset$; also the Levi graph of a coset geometry is defined as the Levi graph of its underlying incidence geometry.

Let $\Gamma=(G:G_0,G_1)$ be a rank two geometry.  The group $G_{01} = G_0 \cap G_1$ is called the \textit{Borel subgroup} of the coset geometry $\Gamma=(G:G_0,G_1)$.  
We call ${\mathcal L}(\Gamma):=\{ G_0, G_1, G_{01}\}$ the \textup{sublattice}
(of the subgroup lattice of $G$) \textit{spanned} by the collection $(G_i)_{i\in \{0,1\}}$.
The elements of the lattice are called the \textit{parabolic subgroups} and the subgroups $G_i$'s are the 
\textit{maximal parabolic subgroups.}  (see \cite{Tit62a}).

Two coset geometries $\Gamma=(G:G_0,G_1)$ and $\Gamma'=(H:H_0,H_1)$ are \textit{isomorphic} if there exists a group isomorphism $f:G \rightarrow H$ such that $f(G_0) = H_0$ and $f(G_1) = H_1$.  A coset geometry $\Gamma=(G:G_0,G_1)$ is \textit{connected} if its underlying incidence geometry is connected. 

We provide an example of a connected rank two coset geometry  (see \cite{DS2010}).
There exist exactly two  geometries  $\Gamma (PSL(2, 31): D_{15}, A_5) $ up to conjugacy, and one up to isomorphism. 
This geometry is connected since $\langle  D_{15}, A_5 \rangle =PSL(2, 31)$.
The subgroups $D_{15}, A_5$ are the maximal parabolic subgroups and the borel subgroup is $D_5$.

\section{Background results}

Here we present background results regarding graphs and incidence geometries that are necessary to understand our results.  Many of the results are part of the folklore, so proofs will often be omitted. 

\subsection{Edge-transitive graphs}
It is well-known (see for instance~\cite{Big93}) that a group $G$ that acts transitively on the set of edges of a graph $X$ has at most two vertex orbits. We state this as a proposition.

\begin{proposition}\label{propC}~\cite{Big93}
If $G$ is a group that acts transitively on the set of edges of a connected graph $X$, then there are at most two orbits on the vertex set.  If there are two orbits, then $X$ is bipartite. If there is one orbit, then $X$ is $G$-vertex-transitive and $G$ acts transitively on the set of arcs of $X$. 
\end{proposition}

In this paper we are interested mainly in the case where the action has two orbits: the orbit of black vertices and the orbit of white vertices. If $G$ is too big, i.e. if there is only one orbit, then we select an appropriate subgroup of $G$, if it exists.

\begin{proposition} \label{propD}
Let $X$ be a $G$-edge-transitive connected graph.
 Then exactly one of the following is true.
\begin{enumerate}
\item $X$ is $G$-arc transitive.
\item $X$ is $G$-half-arc transitive. (edge- and vertex-transitive but not arc-transitive).
\item $X$ is $G$-semi-symmetric. (edge-transitive but not vertex-transitive.)
\end{enumerate}
\end{proposition}

\begin{proposition}\label{propE} ~\cite{Tut66}
If $X$ is $G$-half-arc transitive then it is regular and has an even valence.
\end{proposition}

\begin{proposition}\label{propF}
Let $X$ be a bipartite graph with a given black and white vertex coloring. Then 
$Aut_o(X) = Aut_O(X)$ if and only if no automorphism of $X$ maps a black vertex to a white vertex.
\end{proposition}

\begin{proof}
This follows from the definition of $Aut_o(X)$ and $Aut_O(X)$.
\end{proof}

\begin{proposition} \label{propG}
Let $X$ be a connected bipartite graph with a given black and white coloring. Then $Aut(X) = Aut_O(X)$.
A disconnected bipartite graph has $Aut(X) = Aut_O(X)$ if and only if each connected component $Y$ has
$Aut_o(Y) = Aut_O(Y)$.
\end{proposition}

Hence, for connected bipartite graphs we need to consider only $Aut_o(X)$ and $Aut(X)$.

\begin{proposition}\label{propH}
If $X$ is a connected bipartite graph with a given black and white coloring, then either
$Aut_o(X) = Aut(X)$ or $Aut_o(X)$ is a subgroup of index 2 in $Aut(X)$ and the graph is self-dual.
\end{proposition}

\subsection{Rank two incidence- and coset geometries}

From now on all incidence and coset geometries that we consider are of rank two. Many propositions in what follows rely on this restriction.

\begin{proposition}\label{propJ}
A rank two incidence geometry that is flag-transitive or weakly flag-transitive is also point- and line-transitive.  
\end{proposition}

\begin{proof}
By definition flag-transitive implies point- and line-transitive. 
Suppose that the incidence geometry is weakly flag-transitive, in other words $Aut (X)$ is edge-transitive and $Aut_o(X)$
is not edge-transitive. Following~\cite{MarusicPisanski}, weakly flag-transitive is equivalent to half-arc-transitive; 
which implies point- and line-transitive.
\end{proof}

The following proposition gives a characterization of connected rank two coset geometries.

\begin{proposition}\label{propK}
The coset geometry $\Gamma=(G:G_0,G_1)$ is connected if and only 
if $G_0$ and $G_1$ together generate the whole group $G$.
\end{proposition} 

\begin{proof}
This result  is a particular case of Theorem 2 in~\cite{Deh94}.
\end{proof}

In general, the number of connected components of a coset geometry is equal to the index of $\langle G_0,G_1 \rangle$ in $G$.  If $G_0,G_1$ do not generate the whole group $G$, we may take the group $H = \langle G_0,G_1\rangle$ and the resulting geometry $\Gamma'=(H:G_0,G_1)$ is a connected component of the original geometry.  From now on we assume that we are dealing with connected geometries.

\begin{proposition}\label{propL}~\cite{Tit74}
Let $\Gamma=(G:G_0,G_1)$ be a coset geometry. The corresponding incidence geometry $\Psi$ is $G$-flag-transitive.
\end{proposition}

\section{Characterizations of flag-transitive incidence geometries}

In this section, we give two characterizations of flag-transitive incidence geometries. 
Both of the following propositions are provided by Tits in 1863 (see for example~\cite{Tit62a}).  
They can be seen as versions of what is known as the ``Tits' Algorithm", which has been developed by Francis Buekenhout (see for example~\cite{BCDL2001,Deh94} and references).

\vspace{0.2cm}

\begin{proposition}~\label{propM}
Let $\Psi = (P,L,I)$ be a connected flag-transitive incidence geometry. Then there exists a group $G \leq Aut_o(\Psi)$ and two 
subgroups $G_0$ and $G_1$ with the following three properties.
\begin{enumerate}
\item $G$ acts flag-transitively on $\Psi$.
\item $G_0$ and $G_1$ together generate $G$.
\item The underlying incidence geometry of the coset geometry $\Gamma=(G:G_0,G_1)$ is isomorphic to $\Psi$.  
\end{enumerate} 
\end{proposition}

The points of $\Psi$ can be labeled by left cosets of $G_0$ in $G$, the lines of $\Psi$ can be labeled as left cosets of $G_1$ in $G$, and the flags of $\Psi$ can be labeled by left cosets of $G_{01}=G_0 \cap G_1$ in $G$. 

\vspace{0.2cm}

\begin{proposition}~\label{propN}
Let $X$ be a connected bipartite edge-transitive graph with a given black and white coloring such that $Aut_o(X)$ acts transitively on the edges of $X$, and let $b$  be a  black vertex incident to a white  vertex $w$.
For every subgroup $H$ of $Aut_o(X)$ that acts transitively on the edges of $X$, there exists a coset geometry $\Gamma=(H:H_0,H_1)$
 whose underlying incidence geometry has Levi graph isomorphic to $X$, with $H_0=H \cap Stab(Aut_o(X),b)$, and $H_1=H \cap Stab(Aut_o(X),w)$.
\end{proposition}

Using the above proposition, we see that from a flag-transitive incidence geometry $\psi$, we can define many associated coset geometries. 
One of these coset geometries is the largest core-free coset geometry, which we denote by $\Gamma_{\Psi}$.
To constuct these geometry $\Gamma_{\Psi}$, choose any black vertex $b$ and adjacent white vertex $w$, $G=Aut_o(X)$, $G_0=Stab(G,b)$, $G_1=Stab(G,w)$. Any other black vertex can be chosen, which gives an isomorphic construction.

 Now we define a \textit{stable} coset geometry $\Gamma= (H:H_0,H_1)$ to be a coset geometry with the property that
 $\Gamma_{\Psi (\Gamma)} \cong \Gamma$.  In~\ref{examplehemi}, we give an example to show that not all coset geometries are stable.
We also could define a \textit{stable} edge-transitive incidence geometry to be with the property $\Psi(\Gamma_{\Psi}) \cong\Psi$, but we notice that all edge-transitive incidence geometries are stable.  In other words, by combining the two above propositions, we observe that if we start with a connected flag-transitive incidence geometry, and we construct an associated coset geometry, the incidence geometry of this coset geometry is the same as the original one.

Each coset geometry $\Gamma$ defines a unique incidence geometry $\Psi (\Gamma)$, but in general, the converse is false. This leads us to the following proposition in the connected case.

\begin{proposition}
If $\Psi$ is a connected rank two incidence geometry, then there is either no coset geometry that has $\Psi$ as its underlying incidence geometry, or there are infinitely many non-isomorphic coset geometries that have $\Psi$ as their underlying incidence geometry.
\end{proposition}

\begin{proof}
Let $\Psi$ be a connected rank two incidence geometry.  By Proposition~\ref{propL}, if $\Psi$ is not flag transitive, then there is no coset geometry that has $\Psi$ as its underlying incidence geometry.  Assume that $\Psi$ is flag transitive.  By Proposition~\ref{propM}, there is a coset geometry $\Gamma_{\Psi}=(G:G_0,G_1)$ with underlying incidence geometry isomorphic to $\Psi$.  Now let $\Gamma'=(G\times H:G_0 \times H ,G_1 \times H)$, which is not isomorphic to $\Gamma$ as a coset geometry (as long as $H$ is nontrivial).  However, they yield the same incidence geometry $\Psi$. 
\end{proof}

\vspace{0.2cm}

\begin{proposition}~\label{propO}
Let $\Gamma=(G:G_0,G_1)$ be a connected coset geometry and $\Psi$ its incidence geometry.  The $G$-stabilizer of any point of $\Psi$ is isomorphic to $G_0$ and the $G$-stabilizer of any line of $\Psi$ is isomorphic to $G_1$. The stabilizer of each flag (fixing both endpoints) is isomorphic to $G_{01} = G_0 \cap G_1$. 
\end{proposition}

\begin{proof}
Following Proposition~\ref{propL}, the incidence geometry $\Psi$ is $G$-flag-transitive, which by Proposition~\ref{propN} implies that the $G$-stabilizer of any point of $\Psi$ is isomorphic to $G_0$ and also that the $G$-stabilizer of any line of $\Psi$ is isomorphic to $G_1$.

Let $G_{01}= Stab(Aut_o(X), \{p,l\})$, be the stabilizer of a flag. Assume now that $G'_{01} = G_0 \cap G_1$.
If $\alpha \in G'_{01}$ then $\alpha(p)=p$ and $\alpha(l)=l$ which implies that $\alpha(\{p,l\})=\{p,l\}$. Hence $G'_{01} \leq G_{01}$. On the other side, if $\alpha \in G_{01}$  then $\alpha(\{p,l\})=\{p,l\}$. Since the automorphisms are color preserving it implies that $\alpha(p)=p$ and $\alpha(l)=l$. Hence $G_{01} \leq G'_{01}$.
\end{proof}
The proof of this proposition implies that  if we require that $G \cong G'$, then two coset geometries $\Gamma=(G:G_0,G_1)$ and $\Gamma'=(G':G_0',G_1')$, 
with the same incidence geometries, are always isomorphic.\\

Given the Levi graph $X$ of a coset geometry $\Gamma=(G:G_0,G_1)$, the group $G$ is not always isomorphic to a subgroup of $Aut(X)$.  There is still a natural action of $G$ on $X$; however this action is not always faithful.  The following theorem gives a necessary and sufficient condition for a faithful action.

\begin{theorem}
Let $\Gamma=(G:G_0,G_1)$ be a connected coset geometry with Levi graph $X$; then $G_0\cap G_1$ is core-free in $G$ if and only if $G$ acts faithfully on the edge set of $X$.
\end {theorem}
\begin{proof}
In~\cite{GLP2004} it is proved that if $G_0\cap G_1$ is core-free in $G$, then $G$ acts faithfully on the cosets of $G_0$ and the cosets of $G_1$ in $G$, and thus acts faithfully on the edges of $X$.  To prove the converse, 
let $G_{01}=G_0\cap G_1$ and assume $G_{01}$ is not core-free in $G$.  This implies there exists a nontrivial element $g \in \underset{f \in G}{\bigcap} f G_{01} f^{-1}$.  Thus, this $g$ fixes left cosets $fG_{01}$ for all $f \in G$.  As we saw in Proposition~\ref{propM}, these left cosets determine the edges of $X$.  Therefore $g$ fixes all the edges of $X$, and the action of $G$ on $X$ is not faithful.
\end{proof}

We denote any rank two coset geometry $\Gamma=(G:G_0,G_1)$ with the property that $G_0\cap G_1$ is core-free in $G$ as a \textit{core-free coset geometry}.

\vspace{0.2cm}

\begin{corollary}~\cite{GLP2004}
Let $X$ be the Levi graph of a core-free coset geometry $\Gamma=(G:G_0,G_1)$.  Then $G$ is a subgroup of $Aut_0(X)$, and $X$ is $G$-edge transitive and $G$-vertex intransitive.  
\end{corollary}

\subsection{A non-stable coset geometry. \label{examplehemi}}  

As we have noted, incidence geometries are more general objects than coset geometries.  Moreover, if an incidence geometry is flag-transitive it gives rise to a coset geometry.  However, it can happen that many non-isomorphic coset geometries have the same associated incidence geometry.  We have already seen this to be the case where $G$ is too large to be isomorphic to a subgroup of $Aut_o(\Psi)$.  It can also happen, if $\Psi$ is the associated incidence geometry for a coset geometry $\Gamma=(G:G_0,G_1)$, that $G$ is a proper subgroup of $Aut_o(\Psi)$, and thus the coset geometry constructed using Proposition~\ref{propN} is not isomorphic to $\Gamma$. 

To demonstrate this, we consider the 1-skeleton of the hemidodecahedron as a coset geometry (see for example \cite{Bue85}).  More precisely, let $G \cong Alt(5)$ be generated by $s_0 = (2,3)(4,5)$, 
$s_1= (1,2)(3,4)$, and $s_2= (2,5)(3,4).$  If we let $G_0=\langle s_1,s_2 \rangle \cong S_3$ and $G_1=\langle s_0,s_2  \rangle \cong 2^2$, then the incidence graph for the coset geometry $\Gamma=(Alt(5):G_0,G_1)$ is isomorphic to the subdivision of the Petersen graph.  Thus, $Aut_o(\Psi) \cong Sym(5)$, and we can construct a new coset geometry $\Gamma'=(Sym(5),H_0,H_1)$ as demonstrated in Proposition~\ref{propN}.

\subsection{Isomorphism and Conjugacy}
We can choose different notions of what it means for geometries to be equivalent.  We have defined isomorphisms of incidence geometries as a color-preserving isomorphisms of bipartite graphs.  For two coset geometries $\Gamma=(G:G_0,G_1)$ and $\Gamma'=(G':G_0',G_1')$ to be equivalent, it must be the case that the groups $G \cong G'$.  \\
Then, we have at least one new way of defining equivalence between coset geometries.  We recall that two coset geometries $\Gamma=(G:G_0,G_1)$ and $\Gamma'=(G:G'_0,G'_1)$ are defined to be isomorphic if there is an element of the automorphism group of $G$ that sends $G_0$ and $G_1$ to $G_0'$ and $G_1'$ respectively.  On the other hand, we define two coset geometries $\Gamma$ and $\Gamma'$ to be conjugate if there is an element of the inner automorphism group of $G$ that sends $G_0$ and $G_1$ to $G_0'$ and $G_1'$ respectively.
\begin{proposition}
There exits isomorphic coset geometries which are not conjugate.
\end{proposition}
\begin{proof}
We prove this proposition by providing the following example (see~\cite{Bue82}).  Let $\Gamma=(M_{12}:K(99),K(95))$ and $\Gamma'=(M_{12}:\overline{K(99)},h^{-1}K(95)h)$; here $K(99)$ is isomorphic to $M_9 \ltimes S_3$, $\overline{K(99)}$ is isomorphic to a non-conjugate copy of $M_9 \ltimes S_3$, $K(95)$ is isomorphic to $M_8 \ltimes S_4$, and $h$ is any element of $M_{12}$.  In this example $\Gamma$ and $\Gamma'$ are isomorphic, but not conjugate.  Although $G_1$ and $G_1'$ are conjugate in $G$, $G_0$ and $G_0'$ are not. 
\end{proof}

\section{Geometric properties of edge-transitive bipartite graphs}

For a given bipartite, edge-transitive graph $X$ with a fixed black and white coloring, we may try to determine all subgroups $H$ of $Aut_o(X)$ that act transitively on the edge set of $X$.   As we have seen earlier, there are three more groups of interest here: $H_0,H_1,H_{01}$: the stabilizer 
in $H$ of a black vertex, of a white vertex, and of an edge, respectively. 

There are many axioms to verify on coset geometries (see for example \cite{BCDL2001}), therefore it is important to investigate the relations between the groups  $H_0,H_1,H_{01}$ and $H$ for each coset geometry. 
Indeed, for each transitive subgroup $H$ we ask the following:

\begin{enumerate}
\item Is $H_{01}$ maximal in $H_0$? Does $H_0$ act doubly transitively on the set of neighbors of a black vertex?
\item Is $H_{01}$ maximal in $H_1$? Does $H_1$ act doubly transitively on the set of neighbors of a white vertex?
\item Is $H_0$ maximal in $H$? Does $H$ act doubly transitively on the set of black vertices?
\item Is $H_1$ maximal in $H$? Does $H$ act doubly transitively on the set of white vertices?
\item Is $H$ normal in $Aut_o(X)$?
\item Is $H_0$ normal in $Aut_o(X)$?
\item Is $H_1$ normal in $Aut_o(X)$?
\item Is $H_{01}$ normal in $Aut_o(X)$?
\item Is $X$ $H$-self-dual?  If so, what is the smallest order of a duality? 
\item What is the maximum order of an element of $H$?
\end{enumerate}

The answer to these questions for many graphs 
can be found online at:
\begin{center}  http://dev.ulb.ac.be/solvay/math/geomgr.html \end{center}

\section{Lists of edge-transitive graphs}

The most well-known census of arc-transitive trivalent graph is the Foster Census~\cite{BouwerFoster}. 
It contains arc-transitive cubic graphs and has been recently vastly extended by Marston Conder~\cite{ConderFoster}. 
The case of semi-symmetric cubic graph was studied in~\cite{ConderSemi3}.
Iofinova and Ivanov~\cite{II} showed that there exist exactly five bipartite cubic semisymmetric graphs whose automorphism groups preserves the bipartite parts and acts primitively on each part.  
For each bipartite graph $X$ in these  censuses, we have constructed all possible core-free coset geometries with Levi graph isomorphic to $X$.

Quartic edge-transitive graphs have been extensively studied by Poto\v{c}nik in ~\cite{primoz4valent}, and a census of these graphs can be found in~\cite{pw4census}.  For each bipartite graph $X$ in this census which has a vertex stabilizer of size less than 100,  we have constructed all possible core-free coset geometries with Levi graph isomorphic to $X$.

\section{Rank two geometries and configurations of points and lines}
Rank two geometries that have semi-regular bipartite Levi graphs of girth at least 6 are also known as combinatorial
configurations of points and lines. For such geometries a natural question to ask is whether they admit geometric realization as configurations of points and lines in the Euclidean plane. Edge-transitive bipartite Levi graphs give some promises that methods of polycyclic configurations~\cite{BobenPisanski} could apply in many cases. The case of quartic half-arc-transitive Levi graphs has been studied in~\cite{MarusicPisanski}. For basics of geometric configurations of points and lines the reader is referred to Gr\" unbaum~\cite{Grunbaum}. 

An interesting question arises from the study of cyclic configurations; see, for instance,~\cite{Grunbaum,HMP}.
\begin{problem}
Determine all cubic and quartic flag-transitive cyclic configurations. 
\end{problem}

\section{Experimental result}\label{results}

Using Magma~\cite{BCP97}, for each graph $X$ in various families of edge transitive graphs, a number of coset geometries are constructed using edge transitive subgroups $H$ of $Aut_o(X)$.  The groups $H_0$, $H_1$, and $H_{01}$ are constructed as in Proposition~\ref{propN}. For each coset geometry we are interested in the relationships between the groups.  Some of this information is contained in the columns labeled ``Max" and ``Norm."

{The ``Max" column has a sequence of four elements, corresponding to the four group relationships $H_{01} \leq H_0$, $H_{01} \leq H_1$, $H_{0} \leq H$, and $H_{1} \leq H$ respectively.
The ``Max" column has a sequence of four elements from the possibilities: ``M" which stands for Maximal but not two transitive, ``2T" which stands for 2-transitive, and ``X" which stands for neither maximal nor 2-transitive.
The ``Norm" column has a sequence of four elements, corresponding to the four group relationships $H \leq Aut_o(X)$, $H_0 \leq Aut_o(X)$, $H_1 \leq Aut_o(X)$, and $H_{01} \leq Aut_o(X)$ respectively.
The ``Norm" column tells which of these four subgroups are normal.  The ``Dual" column shows if the coset geometry is self-dual, and if so the minimal order of a duality.  Finally, the ``MaxOrd" column tells the maximum order of an element in $H$. 

Here, we provide further analysis of the ``Max" and ``Norm" columns (for both 3- and 4-valent symmetric graphs), giving the number of times each outcome appears, and the structure of groups in unique cases.
We also give two explanations of why some possibilities in this column do not appear. 

\vspace{.5cm}
\begin{center}
\begin{tiny}
\begin{tabular}{cccc}
\begin{tabular}{|c|c|} \hline
3-valent Sym.& Max\\ \hline
2T2TXX &315 \\ \hline
2T2T2T2T & 4 \\ \hline
 MMXX & 587 \\ \hline
 M2TXX & 32 \\ \hline
  2T2TMM & 21 \\ \hline
  M2T2TM & 1  \\ \hline
  2TMXX & 32 \\ \hline
  2TMM2T &  1 \\ \hline
  MMMM & 88 \\ \hline
  MM2T2T & 1 \\ \hline
\end{tabular}
&
\begin{tabular}{|c|c|} \hline
3-valent Sym.& Norm\\ \hline
YNYN&1 \\ \hline
YYYY&1 \\ \hline
YNNY&674 \\ \hline
NNNY&1 \\ \hline
YYNN&1 \\ \hline
YNNN&404 \\ \hline
\end{tabular}
&
\begin{tabular}{|c|c|} \hline
4-valent Sym.& Max\\ \hline
2TXX2T &11\\ \hline
2TXXM &1\\ \hline
X2TXX &26\\ \hline
2TXMM & 1\\ \hline
2TXXX &19\\ \hline
XX2T2T &1\\ \hline
XXMM &17\\ \hline
2T2TMM &5\\ \hline
X2T2TX &11\\ \hline
2T2TXX &38\\ \hline
XXXX &1503\\ \hline
2T2T2T2T &10\\ \hline
\end{tabular}
&
\begin{tabular}{|c|c|} \hline
4-valent Sym.& Norm\\ \hline
YNNN & 466\\ \hline
NNYN  & 4\\ \hline
YNNY &525\\ \hline
NYNN&4\\ \hline
NNYY&3\\ \hline
NYNY&3\\ \hline
YYNN&2\\ \hline
YYYY&1\\ \hline
YNYN&2\\ \hline
NNNN&365\\ \hline
NNNY&268\\ \hline
\end{tabular}
\end{tabular}
\end{tiny}
\end{center}

\vspace{.5cm}

In the 3-valent symmetric case, if we look at all possible sequences of four elements in the ``Max" column, we can rule out  the sequences with an X in the first or second position.
Indeed, these are $3$-valent graphs, which implies that $H_{01}$ is always maximal in $H_0$ and in $H_1$, since 3 is prime.

In the census of 3-valent symmetric graphs, three of the ten possible sequences in the ``Max" column only appear once.
Two of them (``2TMM2T" and ``M2T2TM") arise from coset geometries that are duals of each other:
$(S_3 \times C_3 : C_3\times 2, S_3)$ and $(S_3 \times C_3 :  S_3,C_3\times 2)$, with $H_{01}=C_2$. 
These geometries arise from the complete bipartite graph $K_{3,3}$.
The third possibility ``MM2T2T", comes from the cubical graph, giving the unique geometry $(A_4: C_3,C_3) $.
We note that none of these three sequences appears in the 4-valent case.

Looking at the 1643 geometries that we analyze from 4-valent symmetric graphs, we make the observation that M cannot appear as one of the first 2 elements in the ``Max" column. 
This follows from the fact that the number of cosets of $H_{01}$ in $H_i$ with $i=0,1$ equals four; also if $H_{01}$ is a maximal subgroup of $H_i$ with $i=0$ or $1$, then $H_i$ acts primitively on the cosets of $H_{01}$.
There are only two primitive groups acting on 4 points, namely $A_4$ and $S_4$, both of which are 2-transitive groups.

There are three sequences that appear only once in the ``Max" column.
The sequence ``XX2T2T" arises from the graph C4[10.2] and gives the coset geometry $((C_5:C_4): C_4, C_4)$.

The two others arise from the graph C4[110.7] giving the coset geometries $(PGL(2,11) : S_4, D_{6})$ and $(PSL(2,11) : A_4, D_{3})$.\\

When ``YYYY" appears in the ``Norm" column there is a unique geometry
in our list coming from a 4-valent graph, and a unique geometry in our
list coming from a 3-valent graph.  The graph ``C4[8.1]" in our tables
is the complete bipartite graph $K_{4,4}$.  It gives rise to the coset
geometry $(C_2^2 \times C_2^2: C_2^2, C_2^2)$.

Similarly, the complete bipartite graph $K_{3,3}$ gives rise to a few
coset geometries that are unique in the ``Norm" column $(C_3 \times
C_3: C_3, C_3)$ for ``YYYY", $(S_3 \times C_3: S_3, C_3 \times 2)$ for
``YYNN", and $(S_3 \times C_3: C_3 \times 2, S_3)$ for ``YNYN".

When ``NNNY" appears in the ``Norm" column there is a unique geometry
in our list coming from a 3-valent graph.
The Heawood graph C14.1 provides the unique coset geometry in our
table $((C_7:C_3):C_3,C_3)$.\\

We also point out that the ``Max" column does not fix the ``Norm" column (and vise versa).  This can be seen through the previous examples.

\section{Concluding remarks}
In this paper we established an exact connection between rank two core-free coset geometries and two-colored edge-transitive graphs. 
We focused on example of 3 and 4 regular graphs.  When constructing coset geometries from 4-valent edge transitive graphs, we noticed that some of these graphs have abundantly large automorphism groups.  More precisely, the stabilizer of a vertex is very large in these cases.  Many of these graphs are labeled as ``unworthy graphs" in \cite{pw4census}; however there are other graphs with very large automorphism groups that are ``worthy."  These 4-valent graphs, with
large automorphism groups, have been studied in \cite{cptwiceprimevalency}. Also, in \cite{primozgabriel1, primozgabriel2} 
it is shown that if a graph with such a large automorphism group is also arc transitive, then it belongs to a well understood infinite family of graphs.

This leads to an interesting question.
\begin{problem}
Classify the semi-symmetric graphs with large vertex stabilizers.
Similarly, classify core-free rank two coset geometries coming from these graphs.
\end{problem}

\section{acknowledgement}

We acknowledge financial support from the Fonds de la Recherche Scientifique (FNRS), which facilitated this collaboration.  
We also acknowledge ``Communaut\'e Francaise de Belgique-Action de Recherche Concert\'e" et le Fond David et Alice Van Buuren.
This research was supported in part by ARRS grant P1-0294 and the ESF grant EuroGIGA/GReGAS.
Also, we thank Primo\v{z} Poto\v{c}nik for invaluable discussions, and comments on early drafts of this paper.

\end{document}